\documentclass{article}

\usepackage[utf8]{inputenc}
\usepackage{amsthm,amsmath,amsfonts,amssymb}
\usepackage{kotex}
\usepackage{tikz-cd}
\usepackage{comment}
\usepackage{MnSymbol}
\usepackage{graphicx}

\newcommand{\iif}{\text{if }}
\newcommand{\Mod}[1]{\ (\text{mod}\ #1)}

\newtheorem{lemma}{Lemma}[section]
\newtheorem{theorem}{Theorem}[section]

\newtheorem{prop}{Proposition}[section]

\newtheorem{definition}{Definition}[section]
\newtheorem{corollary}{Corollary}[section]
\newcommand{\Mm}{\mathcal{M}}
\newcommand{\Ss}{\mathcal{S}}
\newcommand{\FE}{\mathrm{FE}}
\newcommand{\zz}{\mathbb{Z}}

\newcommand{\cc}{\mathbb{C}}

\newcommand{\hh}{\mathcal{H}}
\newcommand{\rr}{\mathbb{R}}
\newcommand{\qq}{\mathbb{Q}}

\newcommand{\SL}{\mathrm{SL}}
\newcommand{\GL}{\mathrm{GL}}

\newcommand{\QQ}{\mathcal{Q}}

\newcommand{\smat}[4]{\left(\begin{smallmatrix} #1 & #2 \\ #3 & #4 \end{smallmatrix}\right)}
\newcommand{\pmat}[4]{\begin{pmatrix} #1 &#2 \\ #3 & #4\end{pmatrix}}

\title{Maass wave forms, quantum modular forms and Hecke operators}
\author{Seewoo Lee}
\begin{document}

\maketitle

\begin{abstract}
We prove that Cohen's Maass wave form and Li-Ngo-Rhoades' Maass wave form  are Hecke eigenforms with respect to certain Hecke operators. 
As a corollary, we find new identities of the $p$-th coefficients of these Maass wave forms in terms of $p$-th root of unity. 
\end{abstract}

%\tableofcontents
\section{Introduction}

In his ``Lost'' Notebook, Ramanujan suggested a lot of surprising formulas without any proofs, which were proved later by other mathematicians. In \cite{an86q}, Andrew conjectured that the coefficients of following $q$-series in the notebook
\begin{align*}
\sigma(q) &= 1+\sum_{n\geq 1} \frac{q^{n(n+1)/2}}{(1+q)(1+q^{2})\cdots (1+q^{n})} = \sum_{n\geq 0} S(n)q^{n} \\
&=1+q-q^{2}+2q^{3}+\cdots + 4q^{45}+\cdots +6q^{1609}+\cdots
\end{align*}
 satisfy 
 \begin{enumerate}
 \item $\lim\sup_{n\to \infty} |S(n)|=\infty$, 
 \item $S(n)=0$ for infinitely many $n$. 
 \end{enumerate}
These remained conjectures until Andrews, Dyson and Hickerson proved this  in  \cite{an88}. The authors proved an identity $S(n) = T_{\mathrm{C}}(24n+1)$,  where $T_{\mathrm{C}}(n)$s are related to the arithmetic of $\mathbb{Q}(\sqrt{6})$, and proved the  conjecture by using the identity. 

In \cite{co88}, using their results, Cohen found connections between the $q$-series and certain Maass wave form. He observed that the coefficients $T_{\mathrm{C}}(n)$ can be interpreted as sum of  Hecke characters on $\mathbb{Q}(\sqrt{6})$, and proved that the complex valued smooth function $u_{\mathrm{C}}(z)$ on the complex upper half plane $\mathcal{H}$  defined as
$$
u_{\mathrm{C}}(z) = \sqrt{y} \sum_{n\equiv 1\Mod{24}} T_{\mathrm{C}}(n) K_{0}\left(\frac{2\pi |n| y}{24}\right) e^{2\pi i n x/24}, \quad z=x+iy\in \mathcal{H}
$$
is a Maass wave form on $\varGamma_{0}(2)$ with an eigenvalue $1/4$ and a nontrivial multiplier system $\nu_{\mathrm{C}}$. (Here $\mathcal{H} = \{ z\in \mathbb{C}\,:\, z = x+iy, \, y > 0\}$ is the complex upper half plane and $K_{0}(y)$ is the $K$-bessel function defined in the following section. The definition of $\nu_{\mathrm{C}}$ is given in the Theorem \ref{cothm}.) 
In his paper, he mentioned that $u_{\mathrm{C}}(z)$ will be a Hecke eigenform with respect to certain Hecke operators, since the coefficients $T_{\mathrm{C}}(n)$ are multiplicative. 
However, we can not just apply the ordinary Hecke operator since $u_{\mathrm{C}}(z)$ has a nontrivial multiplier system.

In \cite{za10}, Zagier defined \emph{quantum modular forms} which are complex valued functions defined on $\mathbb{Q}$ such that their failure to be modular  can be extended as  analytic functions on $\mathbb{R}$, except finitely many points. He gave several examples of quantum modular forms, and in particular, he proved that there exists a quantum modular form $f_{\mathrm{C}}(x)$ associated with Cohen's Maass wave form, using the theory of period functions of Maass wave forms developed by Lewis-Zagier in \cite{le01}.

In a previous paper \cite{le18}, we defined Hecke operator on the space of quantum modular forms with a nontrivial multiplier system, which was previously defined by Wohlfahrt in \cite{wo57} in case of modular forms. 
This Hecke operator has different domain and codomain. In particular, when the operator acts on certain modular forms, it changes the multiplier system. Also, Str\"omberg defined such Hecke operators on the space of Maass wave forms in his Ph.D. thesis \cite{st05} using Wohlfahrt's idea. (He studied computational aspects about Maass wave forms. Also, he concentrated on the special multiplier systems, eta and theta multiplier systems. For details, see \cite{st07}, \cite{st08}.) 

 Using the explicit formula of $T_{\mathrm{C}}(n)$, we prove that the Cohen's Maass wave form $u_{\mathrm{C}}(z)$ is \emph{indeed}  a Hecke eigenform with respect to the Hecke operators defined in  \cite{le18}. 
We also prove that the two kinds of Hecke operators are compatible, i.e. the map attaching quantum modular forms to Maass wave forms is Hecke-equivariant. 
As a corollary, Zagier's quantum modular form $f_{\mathrm{C}}(x)$ is also a Hecke eigenform, which gives us new identities of $T_{\mathrm{C}}(p)$ in terms of $p$-th root of unity:
$$
\pm T_{\mathrm{C}}(\pm p) = (-1)^{\frac{p^{2}-1}{24}}+\frac{1}{2p} \sum_{j=0}^{p-1}\sum_{n=0}^{p-1} (-1)^{n}\zeta_{p}^{(n+1-\frac{p^{2}-1}{24})j}(1-\zeta_{p}^{j})\cdots(1-\zeta_{p}^{nj})
$$
for any prime $p\geq 5$ with $p\equiv \pm 1\Mod{6}$. (See the Corollary 3.3.)

In \cite{co04}, Corson, Favero, Liesinger and Zubairy find another $q$-series $W_{1}(q)$ and $W_{2}(q)$ which satisfy similar properties as $\sigma(q)$ and related to the arithmetic of $\mathbb{Q}(\sqrt{2})$. Using this $q$-series, Li, Ngo and Rhoades defined a corresponding Maass wave form $u_{\mathrm{L}}(z)$ on $\varGamma_{0}(4)$ with an eigenvalue $1/4$ and a nontrivial multiplier system $\nu_{\mathrm{L}}$ which is not a cusp form, and attached a quantum modular form $f_{\mathrm{L}}(x)$ which is defined only on a certain dense subset of $\mathbb{Q}$ due to non-cuspidality of $u_{\mathrm{L}}$ at some cusp. 
We apply our previous arguments for these examples, and obtain new identities for the $p$-th coefficients of Corson-Favero-Liesinger-Zaubairy's $q$-series $T_{\mathrm{L}}(p)$ in terms of $p$-th root of unity.

%This paper is organized as follows. 
%In Section 2, we review known facts. In Section 3 and 4, we give proofs of main results. 

\emph{Acknowledgement}. This is part of the author's M.S. thesis paper.  
The author is grateful to his advisor  Y. Choie for her helpful advice. 
The author is also grateful to J. Lovejoy for his comments via email.

\section{Preliminaries}

\subsection{$q$-Series and Maass wave forms}

\subsubsection{Cohen's Maass wave form}

We do not have any explicit examples of Maass cusp forms on $\varGamma(1)=\mathrm{SL}_{2}(\mathbb{Z})$, although the existence of such a form is guaranteed by the Selberg trace formula (see \cite{he06}). However, there are some known explicit examples of Maass cusp forms on certain congruence subgroups. Here, we introduce Cohen's example in \cite{co88}, which comes from certain Ramanujan $q$-series in his notebook.

In \cite{an88}, the authors proved Andrew's conjecture by showing that  the $S(n)$, the $n$th coefficient of $\sigma(q)$, is related to the arithmetic of the real quadratic field $\qq(\sqrt{6})$. 
To introduce this, we first define $T_{\mathrm{C}}(m)$ for $m\equiv 1\Mod{6}$ to be the excess of the number of inequivalent solutions of the Pell's equation $u^{2}-6v^{2}=m$ with $u+3v\equiv \pm 1\Mod{12}$ over the number of them with $u+3v\equiv \pm 5\Mod{12}$. Note that the two solutions $(u, v)$ and $(u', v')$ are equivalent if $u'+v'\sqrt{6} = \pm(5+2\sqrt{6})^{r}(u+v\sqrt{6})$ for some $r\in \zz$.

\begin{theorem}[Andrew-Dyson-Hickerson, \cite{an88}]
\label{adh}
We have
$
S(n) = T_{\mathrm{C}}(24n+1)
$
for all $n\geq 0$ and a $q$-series identity
$$
\sigma(q) = \sum_{\substack{n\geq 0 \\ |j|\leq n}} (-1)^{n+j}q^{n(3n+1)/2-j^{2}}(1-q^{2n+1}).
$$
 Also, the sequence $T_{\mathrm{C}}(m)$ is completely multiplicative, i.e. for any two coprime integer $m$ and $n$, $T_{\mathrm{C}}(mn)=T_{\mathrm{C}}(m)T_{\mathrm{C}}(n)$ holds. Finally, we have an explicit formula for $T_{\mathrm{C}}(p^{e})$ as
\begin{align}
\label{TCform}
T_{\mathrm{C}}(p^{e}) = \begin{cases} 0 & \iif p\not\equiv 1\Mod{24}, \,\, e\equiv 1\Mod{2} \\
1 & \iif p\equiv 13 \text{ or }19\Mod{24}, \,\, e\equiv 0 \Mod{2} \\
(-1)^{e/2} & \iif p\equiv 7 \Mod{24}, \,\, e\equiv 0 \Mod{2} \\
e+1 & \iif p\equiv 1 \Mod{24}, \,\, T_{\mathrm{C}}(p) = 2 \\
(-1)^{e}(e+1) & \iif p\equiv 1\Mod{24},\,\, T_{\mathrm{C}}(p) =-2
\end{cases}
\end{align}
where $p$ is a prime $\equiv 1\Mod{6}$ or $p = -p'$ for some prime $p'\equiv 5\Mod{6}$.
\end{theorem}

 In \cite{co88}, using the results in \cite{an88}, Cohen defined a Maass cusp form on $\varGamma_{0}(2)$  with a nontrivial multiplier system. 
Cohen showed the $q$-series identity
$$
\varphi(q) = q^{1/24}\sigma(q)+q^{-1/24}\sigma^{*}(q) = \sum_{n\equiv 1\Mod{24}}T_{\mathrm{C}}(n)q^{|n|/24}
$$
where $\sigma^{*}(q)$ is an another $q$-series defined by
$$
\sigma^{*}(q) = 2\sum_{n\geq 1} \frac{(-1)^{n}q^{n^{2}}}{(1-q)(1-q^{3})\cdots(1-q^{2n-1})}.
$$
Also, he defined a  function $u_{\mathrm{C}}:\mathcal{H}\to \mathbb{C}$ as 
\begin{align}
\label{coma}
u_{\mathrm{C}}(z) = \sqrt{y}\sum_{n\equiv 1\Mod{24}}T_{\mathrm{C}}(n)K_{0}(2\pi |n|y/24)e^{2\pi i n x/24}, \quad z=x+iy\in \hh
\end{align}
where $K_{0}$ is the $K$-bessel function
$$
K_{0}(y) = \frac{1}{2}\int_{0}^{\infty} e^{-y(t+t^{-1})/2} \frac{dt}{t}
$$
and proved the following:
\begin{theorem}[Cohen, \cite{co88}]
\label{cothm}
\begin{enumerate}
\item $\Delta u_{\mathrm{C}}(z) = \frac{1}{4}u_{\mathrm{C}}(z)$ for all $z\in \hh$, where 
$$
\Delta = -y^{2}\left(\frac{\partial^{2}}{\partial x^{2}} + \frac{\partial^{2}}{\partial y^{2}}\right)
$$
is the (hyperbolic) Laplacian on $\mathcal{H}$. 
\item $u_{\mathrm{C}}(\gamma z)=\nu_{\mathrm{C}}(\gamma)u_{\mathrm{C}}(z)$ for all $\gamma\in \varGamma_{0}(2)$ and $z\in \hh$, where the multiplier system $\nu_{\mathrm{C}}:\varGamma_{0}(2)\to \mathbb{S}^{1} = \{z\in \mathbb{C}\,:\, |z| = 1\}$ is a homomorphism determined by 
\begin{align*}
\nu_{\mathrm{C}}\left(\begin{pmatrix}1&1\\0&1\end{pmatrix}\right)=\nu_{\mathrm{C}}\left(\begin{pmatrix}1&0\\2&1\end{pmatrix}\right)=\zeta_{24}. 
\end{align*}
\end{enumerate}
\end{theorem}
Note that $\varGamma_{0}(2)$ is generated by two elements $$T = \pmat{1}{1}{0}{1}, \quad R = \pmat{1}{0}{2}{1}.$$
In the paper, Cohen interpreted the coefficients $T_{\mathrm{C}}(n)$ as a sum of Hecke characters on the number field $\mathbb{Q}(\sqrt{6})$ defined as $$
\chi_{C}(\mathfrak{a}) := \begin{cases}
 i^{yx^{-1}}\left(\frac{12}{x}\right) & \iif y\equiv 0\Mod{2} \\
i^{yx^{-1}+1}\left(\frac{12}{x}\right) & \iif y\equiv 1\Mod{2}
\end{cases}
$$
for $\mathfrak{a} = (x+y\sqrt{6}) \subseteq \mathbb{Z}[\sqrt{6}]$. 
Using this, he showed that the $L$-function attached to $u_{\mathrm{C}}(z)$ coincides with the $L$-function of the Hecke character $\chi_{C}$, and proved that $u_{\mathrm{C}}(z)$ satisfies the desired properties by using the functional equation of the $L$-function. 

% He also showed that it gives a 2-dimensional automorphic representation $\rho$ of $\mathrm{Gal}(\overline{\qq}/\qq)$ with $\det(\rho(c))=1$, where $c\in \mathrm{Gal}(\overline{\qq}/\qq)$ is a complex conjugation. 

\subsubsection{Li-Ngo-Rhoades' Maass wave form}
In \cite{li13},  Li, Ngo and Rhoades find similar example as Cohen's example related to the field $\qq(\sqrt{2})$, based on the Corson, Favero, Liesinger and Zubairy's $q$-series. 
In \cite{co04}, the authors proved that coefficients of  two $q$-series 
\begin{align*}
W_{1}(q)&:= \sum_{n\geq 0} \frac{(q)_{n}(-1)^{n}q^{\frac{n(n+1)}{2}}}{(-q)_{n}} = 1-q+2q^{2}-q^{3}-2q^{5}+3q^{6}+\cdots, \\
W_{2}(q)&:= \sum_{n\geq 1} \frac{(-1; q^{2})_{n}(-q)^{n}}{(q;q^{2})_{n}} = -2q-2q^{3}+2q^{4}+2q^{6}+2q^{8}-2q^{9}+\cdots
\end{align*}
are related to the Hecke character $\chi_{L}$ on $\mathbb{Q}(\sqrt{2})$, where
$$
\chi_{L}(\mathfrak{a}):= \begin{cases} 1& \iif N(\mathfrak{a})\equiv \pm1\Mod{16} \\
-1 & \iif N(\mathfrak{a}) \equiv \pm 7\Mod{16} \\
0 & \text{otherwise}
\end{cases}
$$
for integral ideals $\mathfrak{a}\subset \mathbb{Z}[\sqrt{2}]$. They proved that the following sum 
$$
T_{\mathrm{L}}(n):=\begin{cases}
 \sum_{\substack{\mathfrak{a}\subset \mathbb{Z}[\sqrt{2}]\\ N(\mathfrak{a}) = |n|}} \chi_{L}(\mathfrak{a}) & \iif n\equiv 1\Mod{8}\\
 0 & \text{otherwise}.
 \end{cases}
$$
of Hecke characters appears as coefficients on $W_{1}(q)$ and $W_{2}(q)$:
\begin{theorem}[Corson-Favero-Liesinger-Zubairy, \cite{co04}]
We have 
$$
qW_{1}(q^{8}) + q^{-1}W_{2}(q^{8}) = \sum_{n\in \mathbb{Z}} T_{\mathrm{L}}(n)q^{|n|}.
$$
Also, the sequence $T_{\mathrm{L}}(n)$ is completely multiplicative, i.e. $T_{\mathrm{L}}(mn) =T_{\mathrm{L}}(m)T_{\mathrm{L}}(n)$ for any two coprime integers $m$ and $n$. At last, we have an explicit formula for $T_{\mathrm{L}}(p^{e})$ as
\begin{align}
\label{TLform}
T_{\mathrm{L}}(p^{e}) = \begin{cases} 
e+1 & \iif T_{\mathrm{L}}(p)=2, p\equiv  1(8) \\ 
(-1)^{e}(e+1) & \iif T_{\mathrm{L}}(p)=-2, p\equiv  1(8) \\ 
(-1)^{\frac{e}{2}} & \iif e\equiv 0(2), p\equiv 5(8) \\
0 & \iif p=2 \text{ or } (e\equiv 1(2), p\equiv 5(8))
\end{cases}
\end{align}
where $p$ is a prime $\equiv 1\Mod{4}$ or $p = -p'$ for some prime $p'\equiv 3\Mod{4}$.
\end{theorem}
 Using their results, following the Cohen's argument, Li, Ngo and Rhoades proved that there exists a corresponding Maass wave form on $\varGamma_{0}(4)$ with a nontrivial multiplier system, which is \emph{not} a cusp form.
They defined a function $u_{\mathrm{L}}:\mathcal{H}\to \mathbb{C}$ as
\begin{align}
\label{lima}
u_{\mathrm{L}}(z)= \sqrt{y}\sum_{n\in \mathbb{Z}} T_{\mathrm{L}}(n)K_{0}(2\pi|n|y/8)e^{2\pi i n x/24}, \quad z = x+iy\in \mathcal{H} 
\end{align}
and proved the following:
\begin{theorem}[Li-Ngo-Rhoades, \cite{li13}]
\label{lithm}
$u_{\mathrm{L}}(z)$ is a Maass wave form on $\varGamma_{0}(4)$ with an eigenvalue $\lambda=\frac{1}{4}$ and a nontrivial multiplier system $\nu_{\mathrm{L}}:\varGamma_{0}(4)\to \mathbb{S}^{1}$. More precisely, 
\begin{enumerate}
\item $\Delta u_{\mathrm{L}}(z) = \frac{1}{4}u_{\mathrm{L}}(z)$ for all $z\in \mathcal{H}$. 
\item $u_{\mathrm{L}}(\gamma z) = \nu_{\mathrm{L}}(\gamma)u_{\mathrm{L}}(z)$ for all $\gamma\in \varGamma_{0}(4)$ and $z\in \mathcal{H}$, where the multiplier system $\nu_{\mathrm{L}}:\varGamma_{0}(4)\to \mathbb{S}^{1}$ is a homomorphism determined by  
$$
\nu_{\mathrm{L}}\left(\pmat{1}{1}{0}{1}\right)= \nu_{\mathrm{L}}\left(\pmat{1}{0}{4}{0}\right) = \zeta_{8}, \quad \nu_{\mathrm{L}}\left(\pmat{-1}{0}{0}{-1}\right) = 1. 
$$
\item $u_{\mathrm{L}}(z)$ vanishes at the cusps $0$ and $\infty$, but does not vanish at the cusp $1/2$. 
\end{enumerate}
\end{theorem}
Note that $\varGamma_{0}(4)$ is generated by three elements
$$
T = \pmat{1}{1}{0}{1}, \quad R' = \pmat{1}{0}{4}{1}, \quad -I = \pmat{-1}{0}{0}{-1}.
$$

\subsection{Quantum modular forms}

Quantum modular forms were first defined by Zagier in his paper \cite{za10}. 
They are functions defined on $\qq$ with modular properties, which are not modular but their failure to be modular can be extended as analytic functions on $\mathbb{R}$, except finitely many points. 
We will use the following definition of the quantum modular form, which is slightly different from the original one in \cite{za10}.

\begin{definition}
Let $N$ be a positive integer, $k\in \frac{1}{2}\zz$ and $\nu$ be a multiplier system on $\varGamma_{0}(N)$. 
Then a function $f:\qq\to\cc$ is a quantum modular form of weight $k$ and multiplier system $\nu$ on $\varGamma_{0}(N)$  if for all $\gamma = \left(\begin{smallmatrix}a&b\\c&d\end{smallmatrix}\right)\in \varGamma_{0}(N)$,the functions
$$
f(x)-(f|_{k,\nu}\gamma)(x)=h_{\gamma}(x),
$$
where $$ (f|_{k, \nu}\gamma)(x)=\nu(\gamma)^{-1}|cx+d|^{-k}f\left(\frac{ax+b}{cx+d}\right), \quad \gamma=\begin{pmatrix} a&b \\ c&d\end{pmatrix}\in \varGamma_{0}(N),$$   can be extended smoothly on $\rr$ except finitely many points $S\subset \qq$. 
We let $\QQ_{k}(\varGamma_{0}(N), \nu)$ be the space of weight $k$ quantum modular forms on $\varGamma_{0}(N)$ with multiplier system $\nu$. 
\end{definition}

For example, recall  the Ramanujan's $q$-series $\sigma(q)$. 
One can check that the following $q$-series identities hold:
\begin{align*}
\sigma(q)&= 1+\sum_{n=0}^{\infty}(-1)^{n}q^{n+1}(1-q)(1-q^{2})\cdots(1-q^{n}) \\
\sigma^{*}(q)&= -2\sum_{n=0}^{\infty} q^{n+1}(1-q^{2})(1-q^{4})\cdots (1-q^{2n}). 
\end{align*}
(The first identity was shown in the Andrew's paper \cite{an86} and the second derived in a similar way by Cohen.) Both series make sense when $|q|<1$ and $q$ is a root of unity. When $q$ is a root of unity, two series are related as 
$$
\sigma(q) = -\sigma^{*}(q^{-1}). 
$$
(For the proof, see \cite{za10}.) Now define $f_{\mathrm{C}}:\mathbb{Q}\to \mathbb{C}$ as
\begin{align}
\label{fCdefn}
f_{\mathrm{C}}(x):= q^{1/24}\sigma(q)= -q^{1/24} \sigma^{*}(q^{-1}), \quad q = e^{2\pi i x}, x\in \mathbb{Q}.
\end{align}
In \cite{za10}, Zagier showed that this function is a quantum modular form of weight 1 with a multiplier system $\nu_{\mathrm{C}}$ on $\varGamma_{0}(2)$. 
\begin{prop}[Zagier, \cite{za10}] $f_{\mathrm{C}}:\mathbb{Q}\to \mathbb{C}$ satisfies the functional equation
\begin{align}
\label{fCfeq}
f_{\mathrm{C}}(x+1) = \zeta_{24}f_{\mathrm{C}}(x), \quad f_{\mathrm{C}}(x) - \zeta_{24}^{-1}\frac{1}{|2x+1|}f_{\mathrm{C}}\left(\frac{x}{2x+1}\right) =  h_{C}(x)
\end{align}
for any $x\in \mathbb{Q}$, where $\zeta_{m}=e^{2\pi i /m}$ is an $m$th root of unity and  $h_{C}:\mathbb{R}\to \mathbb{C}$ is a smooth function on $\mathbb{R}$ which is  real-analytic except at $x=-1/2$. 
\end{prop}
Proof uses the slightly generalized version of  the theory of period functions of Maass wave forms developed  by Lewis-Zagier.  (For details, see \cite{za10}.) By the similar argument, Li-Ngo-Rhoades showed that there exists a quantum modular form corresponds to their Maass wave form. 
By noncuspidality of $u_{\mathrm{L}}(z)$, it can be defined on a certain dense  subset of $\mathbb{Q}$, not the whole $\mathbb{Q}$. 
\begin{prop}[Li-Ngo-Rhoades, \cite{li13}]
Let 
$$
S_{\iota} = \{x\in \mathbb{Q}:\,\gamma x=\iota\text{ for some }\gamma\in \varGamma_{0}(4)\}, \;\iota\in \mathbb{Q}\cup\{\infty\}
$$
be subsets of $\mathbb{Q}$. 
Define $f_{\mathrm{L}}:S_{0}\cup S_{\infty}\to \mathbb{C}$ as
\begin{align}
\label{fLdefn}
f_{\mathrm{L}}(x):= \begin{cases}
q^{1/8}W_{1}(q) & \text{if }x\in S_{0} \\
q^{1/8} W_{2}(q^{-1}) & \text{if }x\in S_{\infty} 
\end{cases}
\end{align}
Then $f_{\mathrm{L}}$ satisfies the functional equation
\begin{align}
\label{fLfeq}
f_{\mathrm{L}}(x+1) = \zeta_{8}f_{\mathrm{L}}(x), \quad f_{\mathrm{L}}(x)-\zeta_{8}^{-1} \frac{1}{|4x+1|}f_{\mathrm{L}}\left(\frac{x}{4x+1}\right) = h_{L}(x)
\end{align}
for any $x\in S_{0}\cup S_{\infty}$, where $h_{L}:\mathbb{R}\to \mathbb{C}$ is a smooth function on $\mathbb{R}$ which is real-analytic except at $x=-1/4$. 
\end{prop}
Note that in \cite{li13}, Li, Ngo and Rhoades proved the following $q$-series identity
$$
W_{1}(q) = \sum_{n=0}^{\infty} \frac{(q;q^{2})_{n}(-q)^{n}}{(-q^{2};q^{2})_{n}}
$$
which makes $W_{1}(q)$ a finite sum whenever $x\in S_{0}$. $W_{2}(q)$ also becomes a finite sum if $x\in S_{\infty}$. 

In the case of modular forms, we can also attach  quantum modular forms which are called \emph{Eichler integrals} (see \cite{br16}, \cite{za10} for some examples).  For general holomorphic automorphic forms on the complex upper-half plane (with complex weight and parabolic cocycles), it is possible to study such integrals as a cohomological point of view. 
(For details, see \cite{br18}.) 

\emph{Remark 1.} In \cite{za10}, there is a minor erratum : we have to take an absolute value on $2x+1$. 

\emph{Remark 2.} In \cite{li13}, there are several problems with the original  definition of $f_{\mathrm{L}}(x)$. It should be a quantum modular form of weight 1 on $\varGamma_{0}(4)$, with a multiplier system $\nu_{\mathrm{C}}$, which is $\nu_{0}$ following the notation in \cite{li13}, not $\nu_{1}$. We fixed this as above.

%%%%%%%%%%%%%%% Hecke operator on the space of qmf %%%%%

\subsection{Hecke operators with nontrivial multiplier system}

In case of modular forms with nontrivial multiplier systems (such as theta function and eta function), Hecke operators acting on such modular forms  were studied by Wohlfahrt in \cite{wo57} first. 
By using the same idea, we can also define Hecke operators on the space of quantum modular forms with a nontrivial multiplier system, and this was studied in \cite{le18}. We interpreted Wohlfahrt's definition in the slightly different way. First, we recall the definition of compatibility of multiplier systems. 

\begin{definition}
Let $\varGamma\leq \SL_{2}(\mathbb{Z})$ be a finite index subgroup of $\SL_{2}(\mathbb{Z})$, $\nu, \nu':\varGamma\to \mathbb{S}^{1}$ be any two multiplier systems, and $\alpha\in \GL_{2}^{+}(\mathbb{Q})$. We say two multiplier systems compatible at $\alpha$ if the function $c_{\nu, \nu'}:\varGamma\alpha\varGamma\to \mathbb{S}^{1}$ defined by 
$$
c_{\nu, \nu'}(\gamma_{1}\alpha\gamma_{2}) = \nu(\gamma_{1})\nu'(\gamma_{2})
$$
is a well-defined function, i.e. for any element $\gamma_{1}\alpha\gamma_{2} = \delta_{1}\alpha\delta_{2}$ in $\varGamma\alpha\varGamma$, we have
$$
\nu(\gamma_{1})\nu'(\gamma_{2}) = \nu(\delta_{1})\nu'(\delta_{2}). 
$$
\end{definition}
It is easy to check that $\nu$ and $\nu'$ are compatible at $\alpha$ if and only if $\nu(\gamma) = \nu'(\alpha^{-1}\gamma\alpha)$ for every  $\gamma\in \varGamma\cap \alpha\varGamma\alpha^{-1}$. 
Using this function, we can define a Hecke operator on the space of quantum modular forms which changes multiplier system. 
\begin{theorem}[Lee, \cite{le18}]
\label{lee}
Let $\alpha\in \GL_{2}^{+}(\qq)$ and $\varGamma\leq\SL_{2}(\zz)$ be a congruence subgroup. 
Let $\{\varGamma\beta_{j}=\varGamma\alpha_{j}\alpha\alpha_{j}'\}_{j\in J}$ be the (finite) set of representatives of orbits $\varGamma\backslash \varGamma\alpha\varGamma$. 
Suppose two multiplier systems $\nu, \nu':\varGamma\to\mathbb{S}^{1}$ are compatible at $\alpha$. Then we have an operator $T_{\alpha, \nu, \nu'}:\QQ_{k}(\varGamma, \nu)\to \QQ_{k}(\varGamma, \nu')$ defined as
$$
T^{\infty}_{\alpha, \nu, \nu'}f=\sum_{j}c_{\nu, \nu'}(\beta_{j})^{-1}f|_{k}\beta_{j}
$$
for $f\in \QQ_{k}(\varGamma, \nu)$.  
Here $|_{k}$ is a weight $k$ slash operator satisfying  $f|_{k}\gamma_{1}\gamma_{2}=(f|_{k}\gamma_{1})|_{k}\gamma_{2}$ for any $\gamma_{1}, \gamma_{2}\in \GL_{2}^{+}(\qq)$. 
\end{theorem}

\begin{proof}
See \cite{le18}.
\end{proof}

Especially, we are interested in the case when $\varGamma = \varGamma_{0}(N)$ and $\alpha = \alpha_{p}:= \smat{1}{0}{0}{p}$ where $p\nmid N$ is a prime. In this case, $\varGamma \cap \alpha\varGamma\alpha^{-1} = \varGamma_{0}(pN)$ and coset representatives of $\varGamma\backslash \varGamma\alpha\varGamma$ can be chosen as 
$$
\beta_{0} = \pmat{1}{0}{0}{p} , \quad\beta_{1} = \pmat{1}{1}{0}{p}, \cdots ,\quad\beta_{p-1}=\pmat{1}{p-1}{0}{p}, \quad \beta_{\infty} = \pmat{p}{0}{0}{1}
$$
(For the proof, see \cite{di05}.)
Using this, we constructed new quantum modular forms in \cite{le18} by applying Hecke operators on  Zagier's quantum modular form $f_{\mathrm{C}}(x)$. 
Recall that the quantum modular form $f_{\mathrm{C}}:\mathbb{Q}\to \mathbb{C}$ is in $\mathcal{Q}_{1}(\varGamma_{0}(2), \nu_{\mathrm{C}})$. 
By using SAGE, we made a program to compute the value of $\nu_{\mathrm{C}}(\gamma)$ for any given $\gamma\in \varGamma_{0}(2)$. We have checked that $\nu_{\mathrm{C}}$ and $\nu'=\nu_{\mathrm{C}}^{p}$ are compatible at $\alpha_p$ for any prime $5\leq p\leq 757$, i.e. $\nu_{\mathrm{C}}(\gamma) = \nu_{\mathrm{C}}(\alpha^{-1}\gamma\alpha)^{p}$ for any $\gamma\in \varGamma_{0}(2p)$ by compute its values on generators of $\varGamma_{0}(2p)$. 
From $c_{\nu_{\mathrm{C}}, \nu_{\mathrm{C}}^{p}}(\beta_{\infty}) = (-1)^{\frac{p^{2}-1}{24}}$ and $c_{\nu_{\mathrm{C}}, \nu_{\mathrm{C}}^{p}}(\beta_{j}) = \zeta_{24}^{pj}$ for $0\leq j\leq p-1$ (this is proved in Chapter 3 - see Lemma 3.2.), we have the following corollary. 

\begin{corollary}
For the Zagier's quantum modular form $f_{\mathrm{C}}\in \mathcal{Q}_{1}(\varGamma_{0}(2), \nu_{\mathrm{C}})$, the function 
$$
T_{p}^{\infty}f_{\mathrm{C}}(x):= T_{\alpha_{p}, \nu_{\mathrm{C}}, \nu_{\mathrm{C}}^{p}}^{\infty} f_{\mathrm{C}}(x)=(-1)^{\frac{p^{2}-1}{24}}f_{\mathrm{C}}(px)+\frac{1}{p}\sum_{j=0}^{p-1}\zeta_{24}^{-pj}f_{\mathrm{C}}\left(\frac{x+j}{p}\right)
$$
is in $ \mathcal{Q}_{1}(\varGamma_{0}(2), \nu_{\mathrm{C}}^{p})$ for any $5\leq p\leq 757$. Here $T_{p}^{\infty} = T_{\alpha_{p}, \nu_{\mathrm{C}}, \nu_{\mathrm{C}}^{p}}^{\infty}$.
\end{corollary}

Actually, in the following chapter, we will show that these are just constant multiples of $f_{\mathrm{C}}(x)$, i.e. $f_{\mathrm{C}}$ is a Hecke eigenform with respect to these operators. In this sense, \emph{these quantum modular forms are not new}.

In case of Maass wave forms, Str\"omberg studied such Hecke operators in his Ph.D. thesis \cite{st05} and other papers \cite{st07}, \cite{st08}, based on Wohlfahrt's idea. The only difference between the Hecke operators on the space of quantum modular forms and Maass wave forms is just normalization by a constant. 

\begin{definition}
Let $\alpha\in \GL_{2}^{+}(\mathbb{Q})$ and $\varGamma\leq \SL_{2}(\mathbb{Z})$ be a congruence subgroup. Let $\{\varGamma\beta_{j} = \varGamma\alpha_{j}\alpha\alpha_{j}'\}_{j\in J}$ be the set of representatives of orbits $\varGamma\backslash \varGamma\alpha\varGamma$. Suppose two multiplier systems $\nu, \nu':\varGamma\to \mathbb{S}^{1}$ are compatible at $\alpha$.
Let $\mathcal{M}_{s}(\varGamma, \nu)$ be a space of Maass wave forms on $\varGamma$ with  an eigenvalue $\lambda = s(1-s)$ and a multiplier system $\nu$. For any $u\in \mathcal{M}_{s}(\varGamma, \nu)$, define 
$$
T_{\alpha, \nu, \nu'}u = \det(\alpha)^{-1/2}\sum_{j\in J} c_{\nu, \nu'}(\beta_{j})^{-1} u|\beta_{j},
$$
where $|$ is a slash operator defined by 
$$
(u|\gamma)(z) := u(\gamma z) = u\left(\frac{az+b}{cz+d}\right), \quad \gamma = \pmat{a}{b}{c}{d}\in \mathrm{GL}_{2}^{+}(\mathbb{Q}).
$$
%which satisfies $u|(\gamma_{1}\gamma_{2}) = (u|\gamma_{1})|\gamma_{2}$ for any $\gamma_{1}, \gamma_{2}\in \GL_{2}^{+}(\mathbb{Q})$.
\end{definition}
By using the same argument in the proof of the Theorem 4.2 in \cite{le18}, we can easily check that the following proposition holds, so we omit the proof. Note that the Laplacian $\Delta$ commutes with the slash operator. 
\begin{prop}
$T_{\alpha, \nu, \nu'}$ is a linear map from $\mathcal{M}_{s}(\varGamma, \nu)$ to $\mathcal{M}_{s}(\varGamma, \nu')$. Also, image of $\mathcal{S}_{s}(\varGamma, \nu)$ is contained in $\mathcal{S}_{s}(\varGamma, \nu')$, where $\mathcal{S}_{s}(\varGamma, \nu)$ is a space of Maass cusp forms in $\mathcal{M}_{s}(\varGamma, \nu)$. 
\end{prop}

%%%%%%%%%%%%%%% Result %%%%%%%%%%%%%%%%

\section{Main results}

\subsection{Cohen's and Li-Ngo-Rhoades' Maass wave forms are Hecke eigenforms}

Now we prove that Cohen's and Li-Ngo-Rhoades' Maass wave forms defined in \cite{co88} and \cite{li13} are Hecke eigenforms with respect to Str\"omberg's Hecke operators. First, we find the candidate $\nu'$ which is the multiplier system compatible with $\nu$ at $\alpha_{p} = \left(\begin{smallmatrix}1&0\\0&p\end{smallmatrix}\right)$. 
\begin{lemma}
Let $p\geq 5$ be a  prime and $\nu':\varGamma_{0}(2)\to \mathbb{S}^{1}$ be a multiplier system. If $\nu_{\mathrm{C}}:\varGamma_{0}(2)\to \mathbb{S}^{1}$ is the multiplier system associated with $u_{\mathrm{C}}(z)$ (see Theorem \ref{cothm}) and $\nu_{\mathrm{C}}$ and $\nu'$ are compatible at $\alpha_{p}  = \smat{1}{0}{0}{p}$, then $\nu' =\nu_{\mathrm{C}}^{p}$. Also, if they are compatible at $\alpha_{p}$, then we have 
\begin{align*}
c_{\nu_{\mathrm{C}}, \nu_{\mathrm{C}}^{p}}(\beta_{j}) &= \zeta_{24}^{pj} \quad (0\leq j\leq p-1),\\
c_{\nu_{\mathrm{C}}, \nu_{\mathrm{C}}^{p}}(\beta_{\infty}) &= (-1)^{\frac{p^{2}-1}{24}}. 
\end{align*}
\end{lemma}
\begin{proof}
Since 
$$
\beta_{\infty}\pmat{1}{1}{0}{1} = \pmat{p}{p}{0}{1} = \pmat{1}{1}{0}{1}^{p} \beta_{\infty}, \quad \beta_{0}\pmat{1}{0}{2}{1} = \pmat{1}{0}{2p}{p} = \pmat{1}{0}{2}{1}^{p}\beta_{0}, 
$$
$\nu'$ should satisfy $\nu'\left(\smat{1}{1}{0}{1}\right) = \nu_{\mathrm{C}}\left(\smat{1}{1}{0}{1}\right)^{p} = \zeta_{24}^{p}$ and $\nu'\left(\smat{1}{0}{2}{1}\right) = \nu_{\mathrm{C}}\left(\smat{1}{0}{2}{1}\right)^{p} =\zeta_{24}^{p}$, hence the only candidate of $\nu'$ is $\nu_{\mathrm{C}}^{p}$. If they are compatible, then 
\begin{align*}
c_{\nu_{\mathrm{C}}, \nu_{\mathrm{C}}^{p}}(\beta_{j}) &= c_{\nu_{\mathrm{C}}, \nu_{\mathrm{C}}^{p}}\left(\alpha_{p}\pmat{1}{1}{0}{1}^{p}\right) = \zeta_{24}^{pj} \\
 c_{\nu_{\mathrm{C}}, \nu_{\mathrm{C}}^{p}}(\beta_{\infty}) &= c_{\nu_{\mathrm{C}}, \nu_{\mathrm{C}}^{p}}\left(\pmat{1}{1}{0}{1}^{\frac{p+1}{2}}\pmat{1}{0}{2}{1}^{-1}\alpha_{p}\pmat{1}{1}{0}{1}^{\frac{p+1}{2}}\pmat{1}{0}{2}{1}^{-1}\right) = (-1)^{\frac{p^{2}-1}{24}}.  
\end{align*}
Note that $24|(p^{2}-1)$ holds for any prime $p\geq 5$.
\end{proof}

Using this, we define the Hecke operators $T_{p}$ for $p\geq 5$. For some technical reasons, we have to divide into two cases: $p\equiv 1\Mod{6}$ and $p\equiv -1\Mod{6}$. 

\begin{definition}
For a prime $p\geq 5$, define the function $T_{p}u_{\mathrm{C}}:\mathcal{H}\to \mathbb{C}$ as
$$
T_{p}u_{\mathrm{C}}(z):= T_{\alpha_{p}, \nu_{\mathrm{C}}, \nu_{\mathrm{C}}^{p}}u_{\mathrm{C}}(z)  = \frac{1}{\sqrt{p}}\left((-1)^{\frac{p^{2}-1}{24}}u_{\mathrm{C}}(pz) + \sum_{j=0}^{p-1}\zeta_{24}^{-pj}u_{\mathrm{C}}\left(\frac{z+j}{p}\right)\right)
$$
For $p\equiv 1\Mod{6}$ and 
$$
T_{p}u_{\mathrm{C}}(z) := (T_{\alpha_{p}, \nu_{\mathrm{C}}, \nu_{\mathrm{C}}^{p}}u_{\mathrm{C}})(-\overline{z})= \frac{1}{\sqrt{p}}\left((-1)^{\frac{p^{2}-1}{24}}u_{\mathrm{C}}(-p\overline{z}) + \sum_{j=0}^{p-1}\zeta_{24}^{-pj}u_{\mathrm{C}}\left(\frac{-\overline{z}+j}{p}\right)\right)
$$
for $p\equiv -1\Mod{6}$. 
\end{definition}
Note that if $u(z)\in \mathcal{M}_{1/2}(\varGamma_{0}(2), \nu)$ for some multiplier system $\nu:\varGamma_{0}(2)\to \mathbb{S}^{1}$, then $u(-\overline{z}) \in \mathcal{M}_{1/2}(\varGamma_{0}(2), \nu^{-1})$.  For example, $T_{23}u_{\mathrm{C}}(z) \in \mathcal{M}_{1/2}(\varGamma_{0}(2), \nu_{\mathrm{C}}^{-23}) = \mathcal{M}_{1/2}(\varGamma_{0}(2), \nu_{\mathrm{C}})$. 
\begin{theorem}
For any prime $p\geq 5$, we have
$$
T_{p}u_{\mathrm{C}}(z) = \begin{cases} 0 & \iif  p\not\equiv  \pm1\Mod{24} \\ T_{\mathrm{C}}(\pm p)u_{\mathrm{C}}(z) &\iif   p\equiv \pm1\Mod{24}\end{cases}
$$
\end{theorem}

\begin{proof}
Define $\mathcal{W}_{n}(z):= \sqrt{y}K_{0}(2\pi |n|y)e^{2\pi i n x}$  for $z=x+iy\in \mathcal{H}$ and $n\in \mathbb{Z}$. Then we can write the Fourier expansion of $u_{\mathrm{C}}(z)$ as
$$
u_{\mathrm{C}}(z) = \sqrt{24} \sum_{n\equiv 1(24)} T_{\mathrm{C}}(n) \mathcal{W}_{n}\left(\frac{z}{24}\right).
$$ 
We can easily check that the function $\mathcal{W}_{n}(z)$ satisfies the equations
$$
\mathcal{W}_{n}(mz) = \sqrt{m}\mathcal{W}_{mn}(z), \quad \mathcal{W}_{n}\left(z+a\right) = \mathcal{W}_{n}(z)e^{2\pi i n a}\,\,(a\in \mathbb{R}).
$$ 
So for $p\equiv 1\Mod{6}$, we have
\begin{align*}
T_{p}u_{\mathrm{C}}(z) = &\sqrt{\frac{24}{p}}\left[(-1)^{\frac{p^{2}-1}{24}}\sqrt{p}\sum_{n\equiv 1(24)} T_{\mathrm{C}}(n) \mathcal{W}_{pn}\left(\frac{z}{24}\right) \right. \\
&\left.+ \frac{1}{\sqrt{p}}\sum_{n\equiv 1(24)} T_{\mathrm{C}}(n) \mathcal{W}_{n/p}\left(\frac{z}{24}\right) \left(\sum_{j=0}^{p-1}\zeta_{24p}^{(n-p^{2})j}\right)\right]. 
\end{align*}
We can easily show that
$$
\sum_{j=0}^{p-1}\zeta_{24p}^{(n-p^{2})j} = \begin{cases} 0 & \iif n\not\equiv p^{2}\Mod{24p}  \\ p & \iif n\equiv p^{2}\Mod{24p}\end{cases}
$$
so the sum is 
\begin{align*}
T_{p}u_{\mathrm{C}}(z) &= (-1)^{\frac{p^{2}-1}{24}}\sqrt{24}\sum_{n\equiv 1(24)}T_{\mathrm{C}}(n)\mathcal{W}_{pn}\left(\frac{z}{24}\right) + \sqrt{24}\sum_{n\equiv p(24)}T_{\mathrm{C}}(pn)\mathcal{W}_{n}\left(\frac{z}{24}\right) \\
&= \sqrt{24}\sum_{n\equiv p(24)}\left((-1)^{\frac{p^{2}-1}{24}}T_{\mathrm{C}}\left(\frac{n}{p}\right)+T_{\mathrm{C}}(pn)\right)\mathcal{W}_{n}\left(\frac{z}{24}\right)
\end{align*}
if we define $T_{\mathrm{C}}(\alpha)=0$ for $\alpha\not\in\zz$. Now recall the formula of $T_{\mathrm{C}}(n)$: since it is multiplicative, $T_{\mathrm{C}}(n) = T_{\mathrm{C}}(p_{1}^{e_{1}})\cdots T_{\mathrm{C}}(p_{r}^{e_{r}})$ where each $p_{i}$ is a prime satisfying $p_{i} \equiv 1\Mod{6}$ or $p_{i} = -q_{i}$ for some prime  $q_{i}\equiv 5\Mod{6}$. (In any case, we have $p_{i}\equiv 1\Mod{6}$. 
Also, we have the formula \eqref{TCform} of $T_{\mathrm{C}}(p^{e})$ as
\begin{align*}
T_{\mathrm{C}}(p^{e}) = \begin{cases} 0 & \iif p\not\equiv 1\Mod{24}, \,\, e\equiv 1\Mod{2} \\
1 & \iif p\equiv 13 \text{ or }19\Mod{24}, \,\, e\equiv 0 \Mod{2} \\
(-1)^{e/2} & \iif p\equiv 7 \Mod{24}, \,\, e\equiv 0 \Mod{2} \\
e+1 & \iif p\equiv 1 \Mod{24}, \,\, T_{\mathrm{C}}(p) = 2 \\
(-1)^{e}(e+1) & \iif p\equiv 1\Mod{24},\,\, T_{\mathrm{C}}(p) =-2. 
\end{cases}
\end{align*}
First assume that $p\equiv 1\Mod{6}$ and $p\not\equiv 1\Mod{24}$. If $n$ is not a multiple of $p$, then $T_{\mathrm{C}}(n/p)=0$ and $T_{\mathrm{C}}(pn) = T_{\mathrm{C}}(p)T_{\mathrm{C}}(n)=0$ by the multiplicative property of $T_{\mathrm{C}}(n)$ and the previous formula. If $n$ is a multiple of $p$, then $n= p^{k} m$ for some $k\geq1$ and $(m, p)=1$. Then  $n$th coefficient of $T_{p}u(z)$ is  
$$
(-1)^{(p^{2}-1)/24}T_{\mathrm{C}}(p^{k-1}m) + T_{\mathrm{C}}(p^{k+1}m) = [(-1)^{(p^{2}-1)/24}T_{\mathrm{C}}(p^{k-1})+T_{\mathrm{C}}(p^{k+1})]T_{\mathrm{C}}(m).
$$
If $k$ is even, then the $n$th coefficient is zero because of the previous formula again. If $k$ is odd, we can easily check that $(-1)^{(p^{2}-1)/24}T_{\mathrm{C}}(p^{k-1}) + T_{\mathrm{C}}(p^{k+1})=0$ for any $p\not\equiv 1\Mod{24}$. Hence every coefficients vanishes and $T_{p}u(z)\equiv 0$. 

Now we check for the second case,  $p\equiv 1\Mod{24}$.  If $n$ is not a multiple of $p$, then the $n$th coefficient of $T_{p}u_{\mathrm{C}}(z)$ is
$
T_{\mathrm{C}}(pn) =T_{\mathrm{C}}(p)T_{\mathrm{C}}(n)
$
by the multiplicativity of $T_{\mathrm{C}}$. If $n$ is a multiple of $p$, we can write $n$ as $n = p^{k}m$ with $k\geq 1$ and $(m, p) =1$. If $T_{\mathrm{C}}(p)=2$, then 
\begin{align*}
T_{\mathrm{C}}\left(\frac{n}{p}\right)+T_{\mathrm{C}}(pn) &= T_{\mathrm{C}}(p^{k-1})T_{\mathrm{C}}(m)+T_{\mathrm{C}}(p^{k+1})T_{\mathrm{C}}(m) = \left(2k+2\right)T_{\mathrm{C}}(m) \\
&= 2T_{\mathrm{C}}(p^{k})T_{\mathrm{C}}(m) =2T_{\mathrm{C}}(n)=T_{\mathrm{C}}(p)T_{\mathrm{C}}(n)
\end{align*}
and if $T_{\mathrm{C}}(p)=-2$, we have
\begin{align*}
T_{\mathrm{C}}\left(\frac{n}{p}\right)+T_{\mathrm{C}}(pn) &= T_{\mathrm{C}}(p^{k-1})T_{\mathrm{C}}(m)+T_{\mathrm{C}}(p^{k+1})T_{\mathrm{C}}(m) = (-1)^{k-1}\left(2k+2\right)T_{\mathrm{C}}(m) \\
&= -2T_{\mathrm{C}}(p^{k})T_{\mathrm{C}}(m) =-2T_{\mathrm{C}}(n)=T_{\mathrm{C}}(p)T_{\mathrm{C}}(n).
\end{align*}
So for every case, the $n$th coefficient of $T_{p}u_{\mathrm{C}}(z)$ is $T_{\mathrm{C}}(p)T_{\mathrm{C}}(n)$ and so $T_{p}u_{\mathrm{C}}(z) = T_{\mathrm{C}}(p)u_{\mathrm{C}}(z)$. 

Similarly, if $p\equiv 5\Mod{6}$, we have 
$$
T_{p}u_{\mathrm{C}}(z) =\sqrt{24} \sum_{n\equiv -p\Mod{24}} \left((-1)^{\frac{p^{2}-1}{24}}T_{\mathrm{C}}\left(-\frac{n}{p}\right) + T_{\mathrm{C}}(-pn)\right) \mathcal{W}_{n}\left(\frac{z}{24}\right)
$$
by using $\mathcal{W}_{n}(-\overline{z}) = \mathcal{W}_{-n}(z)$. We can treat these cases by the same way as $p\equiv 1\Mod{6}$. 
\end{proof}

We can use the same argument for Li-Ngo-Rhoades' Maass wave form.

\begin{lemma}
Let $p\geq 3$ be a  prime and $\nu':\varGamma_{0}(4)\to \mathbb{S}^{1}$ be a multiplier system. If $\nu_{\mathrm{L}}:\varGamma_{0}(4)\to \mathbb{S}^{1}$ is the multiplier system associated with $u_{\mathrm{L}}(z)$ (see Theorem \ref{lithm}) and $\nu_{\mathrm{L}}$ and $\nu'$ are compatible at $\alpha_{p}  = \smat{1}{0}{0}{p}$, then $\nu' =\nu_{\mathrm{L}}^{p}$. Also, if they are compatible at $\alpha_{p}$, then we have 
\begin{align*}
c_{\nu_{\mathrm{L}}, \nu_{\mathrm{L}}^{p}}(\beta_{j}) &= \zeta_{8}^{pj}\quad(0\leq j\leq p-1), \\
c_{\nu_{\mathrm{L}}, \nu_{\mathrm{L}}^{p}}(\beta_{\infty}) &= 1.  
\end{align*}
\end{lemma}
\begin{proof}
Proof for the first statement is almost same as the Cohen's case.  If they $\nu_{\mathrm{L}}$ and $\nu_{\mathrm{L}}'$ are compatible at $\alpha_{p}$, then 
\begin{align*}
c_{\nu_{\mathrm{L}}, \nu_{\mathrm{L}}^{p}}(\beta_{j}) &= c_{\nu_{\mathrm{L}}, \nu_{\mathrm{L}}^{p}}\left(\alpha_{p}\pmat{1}{1}{0}{1}^{p}\right) = \zeta_{8}^{pj},
\end{align*}
and $c_{\nu_{\mathrm{L}}, \nu_{\mathrm{L}}^{p}}(\beta_{\infty})=1$ follows from the matrix identities:
\begin{align*}
p\equiv 1\Mod{4} \Rightarrow & \,\beta_{\infty}  = \pmat{1}{1}{0}{1}^{\frac{p-1}{4}}\pmat{1}{0}{4}{1}\alpha_{p}\pmat{1}{1}{0}{1}^{-\frac{p-1}{4}}\pmat{1}{0}{4}{1}^{-1} \\
p\equiv -1 \Mod{4} \Rightarrow &\,\beta_{\infty} = -\pmat{1}{1}{0}{1}^{\frac{p+1}{4}}\pmat{1}{0}{4}{1}^{-1}\alpha_{p}\pmat{1}{1}{0}{1}^{\frac{p+1}{4}}\pmat{1}{0}{4}{1}^{-1}.
\end{align*}
\end{proof}

\begin{definition}
For a prime $p\geq 3$, define the Hecke operator $T_{p}$ as
$$
T_{p}u_{\mathrm{L}}(z):= T_{\alpha_{p}, \nu_{\mathrm{L}}, \nu_{\mathrm{L}}^{p}}u_{\mathrm{L}}(z)  = \frac{1}{\sqrt{p}}\left(u_{\mathrm{L}}(pz) + \sum_{j=0}^{p-1}\zeta_{8}^{-pj}u_{\mathrm{L}}\left(\frac{z+j}{p}\right)\right)
$$
for $p\equiv 1\Mod{4}$ and 
$$
T_{p}u_{\mathrm{L}}(z) := T_{\alpha_{p}, \nu_{\mathrm{L}}, \nu_{\mathrm{L}}^{p}}u_{\mathrm{L}}(-\overline{z})= \frac{1}{\sqrt{p}}\left(u_{\mathrm{L}}(-p\overline{z}) + \sum_{j=0}^{p-1}\zeta_{8}^{-pj}u_{\mathrm{L}}\left(\frac{-\overline{z}+j}{p}\right)\right)
$$
for $p\equiv -1\Mod{4}$. 
\end{definition}
Note that $T_{\alpha_{p}, \nu_{\mathrm{L}}, \nu_{\mathrm{L}}^{p}}u(-\overline{z})\in \mathcal{M}_{1/2}(\varGamma_{0}(4), \nu_{\mathrm{L}}^{-p})$ because $T_{\alpha_{p}, \nu_{\mathrm{L}}, \nu_{\mathrm{L}}^{p}}u(z)\in \mathcal{M}_{1/2}(\varGamma_{0}(4), \nu_{\mathrm{L}}^{p})$. 
\begin{theorem}
For any prime $p\geq 5$, we have
$$
T_{p}u_{\mathrm{L}}(z) = \begin{cases} 0 &\iif  p\not\equiv  \pm 1\Mod{8} \\ T_{\mathrm{L}}(\pm p)u_{\mathrm{L}}(z) &\iif  p\equiv \pm 1\Mod{8}\end{cases}
$$
\end{theorem}

\begin{proof}
The proof is almost same as the case of Cohen's Maass wave form. We use multiplicativity and the explicit formula \eqref{TLform} of $T_{\mathrm{L}}(p^{e})$. 
\end{proof}

\subsection{Hecke equivariance}

Now we have two kinds of Hecke operators: on the space of Maass wave forms and quantum modular forms. We show that these operators are compatible, and as a result, we give  new formulas of $T_{\mathrm{C}}(p)$ and $T_{\mathrm{L}}(p)$ in terms of $p$th roots of unity.

In \cite{br07}, by using the Lewis-Zagier theory, Bruggeman showed that one can associate quantum modular forms with Maass wave forms on $\mathrm{SL}_{2}(\mathbb{Z})$ with trivial multiplier systems. 
In \cite{br17}, Bringmann, Lovejoy and Rolen extended Bruggeman's result to the more general situation, in case of Maass wave forms on a congruence subgroup with spectral parameter $s=1/2$ and trivial multiplier systems. By the same argument, we can prove the following result about Maass wave forms with nontrivial multiplier systems. Since the proof is almost same, we  only present the result. 
\begin{theorem}
Let $N$ be a natural number, $I$ be the set of cusps of $\varGamma_{0}(N)$, and $\nu:\varGamma_{0}(N)\to \mathbb{S}^{1}$ be a multiplier system. 
For $u(z)\in \mathcal{M}_{1/2}(\varGamma_{0}(N), \nu)$, let $f:\mathbb{C}\backslash \mathbb{R}\to\mathbb{C}$ be a period function of $u$, i.e. the holomorphic function defined as 
$$
f(z) = \begin{cases}
\int_{z}^{i\infty} [u(\tau), R_{z}(\tau)^{1/2}] &\iif  z\in \mathcal{H} \\
-\int_{\overline{z}}^{i\infty} [R_{z}(\tau)^{1/2}, u(\tau)] &\iif  z\in \mathcal{H}^{-}
\end{cases}
$$
where $\mathcal{H}^{-} = \{z = x+iy\in \mathbb{C}\,:\, y <0\}$,   $[u, v]=[u, v](\tau)$ is a Green's 1-form defined as 
$$
[u, v](\tau) = v \frac{\partial u}{\partial \tau}\hbox{d}\tau + u \frac{\partial v}{\partial \overline{\tau}}\hbox{d}\overline{\tau},
$$
and $R_{z}(\tau)$ is a function defined as
$$
R_{z}(\tau) = \frac{y}{(x-z)^{2}+y^{2}}= \frac{i}{2}\left(\frac{1}{\tau - z} - \frac{1}{\overline{\tau}-z}\right), \quad z\in\mathbb{C}, \tau = x+iy\in \mathcal{H}.
$$ 
Then $f$ can be continuously  extended to a quantum modular form on $\varGamma_{0}(N)$ of weight 1 with multiplier system $\nu$, defined on the subset $\cup_{\iota\in J} S_{\iota}$ of $\mathbb{Q}$, where $J\subseteq I$ is the set of cusps that $u(z)$ vanishes,
and 
$$
S_{\iota} = \{x\in \mathbb{Q}:\, \gamma x = \iota\text{ for some }\gamma\in \varGamma_{0}(N)\}. 
$$
 More precisely, the function $f:\cup_{\iota\in J}S_{\iota}\to \mathbb{C}$ satisfies the functional equations
\begin{align*}
&f(x+1) = \nu(T)f(x), \quad T = \pmat{1}{1}{0}{1}, \\
&f(x) - \nu(\gamma)^{-1} \frac{1}{|cx+d|}f\left(\frac{ax+b}{cx+d}\right) = \int_{-d/c}^{i\infty} [u(\tau), R_{x}(\tau)^{1/2}]
\end{align*}
for every $\gamma = \left(\begin{smallmatrix} a&b\\c&d\end{smallmatrix}\right)\in \varGamma_{0}(N)$ with $c\neq 0$. 
\end{theorem}

Let us denote the above map $\mathcal{M}_{1/2}(\varGamma, \nu)\to \mathcal{Q}_{1}(\varGamma, \nu), \, u\mapsto f$ as $\varPsi_{\nu}$.  We show that $\varPsi_{\nu}$ is a Hecke-equivariant map. 

\begin{theorem} 
\label{comp}
Assume that we can choose coset representatives $\{\varGamma\beta_{j}\}$ of $\varGamma\backslash \varGamma\alpha\varGamma$ that satisfy  $\beta_{j}(i\infty) = i\infty$, i.e. $\beta_{j}$ has a form of $ \left(\begin{smallmatrix}a_{j}&b_{j} \\ 0 & d_{j}\end{smallmatrix}\right)$. Then
the following diagram commutes:
\[
\begin{tikzcd}
\mathcal{M}_{1/2}(\varGamma, \nu) \rar{T_{\alpha, \nu,\nu'}} \dar{\varPsi_{\nu}} & \mathcal{M}_{1/2}(\varGamma, \nu') \dar{\varPsi_{\nu'}} \\
\mathcal{Q}_{1}(\varGamma, \nu) \rar{T^{\infty}_{\alpha, \nu, \nu'}} &\mathcal{Q}_{1}(\varGamma, \nu')
\end{tikzcd}
\]
\end{theorem}

For the proof, we will use the following simple identity about $R_{z}(\tau)$, which can be verified as a direct computation. 
\begin{lemma}
For any $\tau\in \mathcal{H}, z\in\mathbb{C}$ and $g\in \mathrm{GL}_{2}^{+}(\mathbb{R})$, we have 
$$
R_{gz}(g\tau) = \frac{(cz+d)^{2}}{\det(g)} R_{z}(\tau). 
$$
\end{lemma}
\begin{proof}[Proof of Theorem \ref{comp}]
Let $u\in \mathcal{M}_{1/2}(\varGamma, \nu)$ and write $\beta_{j} = \smat{a_{j}}{b_{j}}{0}{d_{j}}$. 
For $z\in \mathcal{H}$, 
\begin{align*}
T_{\alpha, \nu, \nu'}^{\infty} (\varPsi_{\nu} u) (z) &= \sum_{j} c_{\nu, \nu'}(\beta_{j})^{-1} (\varPsi_{\nu}u)|\beta_{j} \\
&= \sum_{j} c_{\nu, \nu'}(\beta_{j})^{-1} \left(\int_{z}^{i\infty} [u(\tau), R_{z}(\tau)^{1/2}]\right)\Big|\beta_{j} \\
&= \sum_{j} c_{\nu, \nu'}(\beta_{j})^{-1} \frac{1}{d_{j}}\int_{\beta_{j}z}^{i\infty} [u(\tau), R_{\beta_{j}z}(\tau)^{1/2}]
\end{align*}
and 
\begin{align*}
\varPsi_{\nu'}T_{\alpha, \nu, \nu'}u(z) &= \int_{z}^{i\infty} [T_{\alpha, \nu, \nu'}u(\tau), R_{z}(\tau)^{1/2}] \\
&= \det(\alpha)^{-1/2} \int_{z}^{i\infty} \left[\sum_{j} c_{\nu, \nu'}(\beta_{j})^{-1} u(\beta_{j}\tau), R_{z}(\tau)^{1/2}\right] \\
&= \det(\alpha)^{-1/2} \sum_{j} c_{\nu, \nu'}(\beta_{j})^{-1}\int_{z}^{i\infty} [u(\beta_{j}\tau), R_{z}(\tau)^{1/2}] \\
&= \det(\alpha)^{-1/2} \sum_{j} c_{\nu, \nu'}(\beta_{j})^{-1} \int_{\beta_{j}z}^{i\infty} [u(\tau), R_{z}(\beta_{j}^{-1}\tau)^{1/2}] \\
&= \det(\alpha)^{-1/2} \sum_{j} c_{\nu, \nu'}(\beta_{j})^{-1} \int_{\beta_{j}z}^{i\infty} [u(\tau), R_{\beta_{j}z}(\tau)^{1/2}]\frac{\det(\beta_{j})^{1/2}}{d_{j}} \\
&= \sum_{j} c_{\nu, \nu'}(\beta_{j})^{-1} \frac{1}{d_{j}}\int_{\beta_{j}z}^{i\infty} [u(\tau), R_{\beta_{j}z}(\tau)^{1/2}]
\end{align*}
since $\det(\beta_{j}) = \det(\alpha)$. Hence we have $T_{\alpha, \nu, \nu'}^{\infty}\varPsi_{\nu} = \varPsi_{\nu'}T_{\alpha, \nu, \nu'}$. We can proceed similarly for $z\in \mathcal{H}^{-}$. 
\end{proof}

Now we will concentrate on Cohen's and Li-Ngo-Rhoades' examples (\eqref{fCdefn} and \eqref{fLdefn}). 
In both cases, the assumption of Theorem \ref{comp} holds, so the diagram commutes. 
For the Cohen's case, since $u_{\mathrm{C}}(z)$ is a Hecke eigenform, $f_{\mathrm{C}}(x)$ is also a Hecke eigenform as a corollary. 

\begin{corollary}
Define the Hecke operator $T_{p}^{\infty}  : \mathcal{Q}_{1}(\varGamma_{0}(2), \nu_{\mathrm{C}}) \to \mathcal{Q}_{1}(\varGamma_{0}(2), \nu_{\mathrm{C}}^{\pm p})$ by 
$$
T_{p}^{\infty}f(z) := T_{\alpha_{p}, \nu_{\mathrm{C}}, \nu_{\mathrm{C}}^{p}}^{\infty}f(z) = (-1)^{\frac{p^{2}-1}{24}} f(pz) + \frac{1}{p} \sum_{j=0}^{p-1} \zeta_{24}^{-pj} f\left(\frac{z+j}{p}\right)
$$
for $p\equiv 1\Mod{6}$ and 
$$
T_{p}^{\infty}f(z) := T_{\alpha_{p}, \nu_{\mathrm{C}}, \nu_{\mathrm{C}}^{p}}^{\infty}f(-z)= (-1)^{\frac{p^{2}-1}{24}} f(-pz) + \frac{1}{p} \sum_{j=0}^{p-1} \zeta_{24}^{-pj} f\left(\frac{-z+j}{p}\right)
$$
for $p\equiv -1\Mod{6}$.
Then $f_{\mathrm{C}}:\mathbb{Q}\to \mathbb{C}$ in \eqref{fCdefn} is a Hecke eigenform, i.e. 
$$
T_{p}^{\infty} f_{\mathrm{C}}(x) = \begin{cases} 0 & \iif   p\not\equiv \pm1\Mod{24} \\ \pm T_{\mathrm{C}}(\pm p)f_{\mathrm{C}}(x) & \iif  p\equiv\pm 1\Mod{24} \end{cases}
$$
for any prime $p\geq 5$ and $x\in \mathbb{Q}$.
\end{corollary}

\begin{proof}
In case of  $p\equiv 1\Mod{6}$, the proof is almost direct. By the Theorem \ref{comp}, we have
\begin{align*}
T_{p}^{\infty}f_{\mathrm{C}}(z) &= (T_{\alpha_{p}, \nu_{\mathrm{C}}, \nu_{\mathrm{C}}^{p}}^{\infty} \circ\varPsi_{\nu_{\mathrm{C}}} u_{\mathrm{C}})(z) \\
&= (\varPsi_{\nu_{\mathrm{C}}^{p}}\circ T_{\alpha_{p}, \nu_{\mathrm{C}}, \nu_{\mathrm{C}}^{p}}u_{\mathrm{C}})(z) \\
&= (\varPsi_{\nu_{\mathrm{C}}^{p}}\circ T_{p}u_{\mathrm{C}})(z) \\
&=\begin{cases}
0 & \iif p\not\equiv 1\Mod{24} \\
T(p)\varPsi_{\nu_{\mathrm{C}}^{p}}u_{\mathrm{C}}(z) = T(p)f_{\mathrm{C}}(z) & \iif p\equiv 1\Mod{24}. 
\end{cases}
\end{align*}
We have to deal with the case $p\equiv -1\Mod{6}$ carefully. For any function $u:\mathcal{H}\to \mathbb{C}$, define an involution $I$ acts on $u$ as $(I\circ u)(z):= u(-\overline{z})$. Similarly, for any function $f:(\mathbb{C}\backslash \mathbb{R})\cup \mathbb{Q}\to \mathbb{C}$,  define another involution $I^{\infty}$ acts on $f$ as $(I^{\infty} f)(z):= f(-z)$. Then we can write $T_{p}= I\circ T_{\alpha_{p}, \nu_{\mathrm{C}}, \nu_{\mathrm{C}}^{p}}$ and $T_{p}^{\infty}= I^{\infty}\circ T^{\infty}_{\alpha_{p}, \nu_{\mathrm{C}}, \nu_{\mathrm{C}}^{p}}$  for $p\equiv -1\Mod{6}$. 
We need a following lemma. 
\begin{lemma}
For any congruence subgroup $\varGamma$ and multiplier system $\nu:\varGamma\to \mathbb{S}^{1}$, we have $I^{\infty}\circ \varPsi_{\nu} = -\varPsi_{\nu^{-1}}\circ I$. 
\end{lemma}
\label{invol}
\begin{proof}
Let $u\in \mathcal{M}_{1/2}(\varGamma, \nu)$. By definition, we have
\begin{align*}
(I^{\infty}\circ \varPsi_{\nu} u)(z) =(\varPsi_{\nu}u)(-z)= \begin{cases}
-\int_{-\overline{z}}^{i\infty} [R_{-z}(\tau)^{1/2}, u(\tau)] & \iif z\in\mathcal{H} \\
\int_{-z}^{i\infty} [u(\tau), R_{-z}(\tau)^{1/2}] & \iif z\in\mathcal{H}^{-}. 
\end{cases}
\end{align*}
We also have
\begin{align*}
(\varPsi_{\nu^{-1}}\circ I u)(z) = \begin{cases}
\int_{z}^{i\infty} [u(-\overline{\tau}), R_{z}(\tau)^{1/2}] & \iif z\in\mathcal{H} \\
-\int_{\overline{z}}^{i\infty} [R_{z}(\tau)^{1/2}, u(-\overline{\tau})] & \iif z\in \mathcal{H}^{-}. 
\end{cases}
\end{align*}
Now for convenience, let $u_{0}(\tau) = u(-\overline{\tau})$. Then for $z\in \hh$, 
\begin{align*}
\int_{z}^{i\infty} [u_{0}(\tau), R_{z}(\tau)^{1/2}] = \int_{z}^{i\infty} \frac{\partial u_{0}(\tau)}{\partial\tau}R_{z}(\tau)^{1/2}d\tau + u_{0}(\tau) \frac{\partial R_{z}(\tau)^{1/2}}{\partial \overline{\tau}} d\overline{\tau}.
\end{align*}
Using the change of variable $\tau' = -\overline{\tau}$ and the identity $R_{z}(-\overline{\tau}) = R_{-z}(\tau)$, we have
\begin{align*}
&\int_{-\overline{z}}^{i\infty} \frac{\partial u_{0}(-\overline{\tau'})}{\partial \overline{\tau'}} \frac{\partial \overline{\tau'}}{\partial \tau} R_{z}\left(-\overline{\tau'}\right)^{1/2}\left(-d\overline{\tau'}\right) + u_{0}\left(-\overline{\tau'}\right) \frac{\partial R_{z}(-\overline{\tau'})}{\partial \tau'}\frac{\partial \tau'}{\partial \overline{\tau}} \left(-d\tau'\right) \\
&=\int_{-\overline{z}}^{i\infty} \frac{\partial u(\tau')}{\partial \overline{\tau'}} R_{-z}(\tau')^{1/2}d\overline{\tau'} + u(\tau) \frac{\partial (R_{-z}(\tau')^{1/2})}{\partial \tau'} d\tau' \\
&= \int_{-\overline{z}}^{i\infty} [R_{-z}(\tau')^{1/2}, u(\tau')] \\
&= -(I^{\infty}\circ \varPsi_{\nu} u)(z).
\end{align*}
We can prove the case when $z\in \mathcal{H}^{-}$ similarly. 
\end{proof}
Now assume $p \equiv -1\Mod{6}$. Then by the Lemma \ref{invol},  
\begin{align*}
T_{p}f_{\mathrm{C}}(z)& = (I^{\infty}\circ T_{\alpha_{p}, \nu_{\mathrm{C}}, \nu_{\mathrm{C}}^{p}}^{\infty}\circ\varPsi_{\nu_{\mathrm{C}}}u_{\mathrm{C}})(z) \\
&= (I^{\infty}\circ \varPsi_{\nu_{\mathrm{C}}^{p}}\circ T_{\alpha_{p}, \nu_{\mathrm{C}}, \nu_{\mathrm{C}}^{p}}u_{\mathrm{C}})(z) \\
&= -(\varPsi_{\nu_{\mathrm{C}}^{-p}}\circ I\circ T_{\alpha_{p}, \nu_{\mathrm{C}}, \nu_{\mathrm{C}}^{p}}u_{\mathrm{C}})(z) \\
&= -(\varPsi_{\nu_{\mathrm{C}}^{-p}}\circ T_{p}u_{\mathrm{C}})(z) \\
&= \begin{cases}
0 &\iif  p\not \equiv -1 \Mod{24} \\
-T(-p)\varPsi_{\nu_{\mathrm{C}}^{-p}}u_{\mathrm{C}}(z) =-T(-p)f_{\mathrm{C}}(z)& \iif p\equiv -1\Mod{24}.
\end{cases}
\end{align*}
\end{proof}

Similarly, we can prove that Li-Ngo-Rhoades' quantum modular form is also a Hecke eigenform. 

\begin{corollary}
Define the Hecke operator $T_{p}^{\infty}:\mathcal{Q}_{1}(\varGamma_{0}(4), \nu_{\mathrm{L}})\to \mathcal{Q}_{1}(\varGamma_{0}(4), \nu_{\mathrm{L}}^{\pm p})$ by 
$$
T_{p}^{\infty} f(z) := T_{\alpha_{p}, \nu_{\mathrm{L}}, \nu_{\mathrm{L}}^{p}}^{\infty}f(z) =  f(pz)+\frac{1}{p}\sum_{j=0}^{p-1}\zeta_{8}^{-pj}f\left(\frac{z+j}{p}\right)
$$
for $p\equiv 1\Mod{4}$ and 
$$
T_{p}^{\infty}f(z) := T_{\alpha_{p}, \nu_{\mathrm{L}}, \nu_{\mathrm{L}}^{p}}^{\infty} f(-z)= f(-pz) + \frac{1}{p}\sum_{j=0}^{p-1} \zeta_{8}^{-pj} f\left(\frac{-z+j}{p}\right)
$$
for $p\equiv -1\Mod{4}$. Then $f_{\mathrm{L}}:S_{0}\cup S_{\infty} \to \mathbb{C}$ is a Hecke eigenform, i.e. 
$$
T_{p}^{\infty}f_{\mathrm{L}}(x) = \begin{cases} 0 & \iif  p\not\equiv \pm1\Mod{8} \\ \pm T_{\mathrm{L}}(\pm p)f_{\mathrm{L}}(x) &  \iif p\equiv \pm1\Mod{8} \end{cases}
$$
for any prime $p\geq 3$ and $x\in S_{0}\cup S_{\infty} $.
\end{corollary}

Actually, we can \emph{see} this using MATLAB. We can plot graphs of the function
$$
h_{L}(x) = f_{\mathrm{L}}(x) - \zeta_{8}^{-1} \frac{1}{|4x+1|}f_{\mathrm{L}}\left(\frac{x}{4x+1}\right)
$$
and
$$
H_{L}(x) = T_{7}^{\infty}f_{\mathrm{L}}(x) - \zeta_{8}^{-1} \frac{1}{|4x+1|} T_{7}^{\infty} f_{\mathrm{L}}\left(\frac{x}{4x+1}\right)
$$
which are both smooth on $\mathbb{R}\backslash \{-1/4\}$. Since $T_{7}^{\infty} f_{\mathrm{L}}(x) = -T_{\mathrm{L}}(-7)f_{\mathrm{L}}(x) = 2f_{\mathrm{L}}(x)$, we should have $H_{L}(x) =2h_{L}(x)$, and we can observe this from the following graph.
Here pink (resp. green) graph shows the real (resp. imaginary) part of $h_{L}(x)$, and red (resp. blue)  graph shows the real (resp. imaginary) part of $H_{L}(x)$. 

\begin{figure}[h]
\centering
\includegraphics[width=0.7\textwidth]{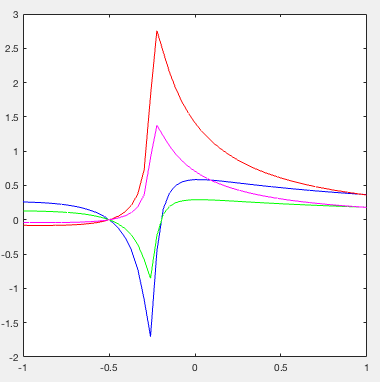}
\caption{Real and imaginary parts of $h_{L}(x)$ and $H_{L}(x)$}
\end{figure}

\subsection{New identities for $T_{\mathrm{C}}(p)$ and $T_{\mathrm{L}}(p)$}

Since $f_{\mathrm{C}}(x)$ is a Hecke eigenform, we can get a new identity for the coefficient $T_{\mathrm{C}}(p)$ by considering the value $T_{p}^{\infty}f_{\mathrm{C}}(0)$.

\begin{corollary} For any prime $p\geq 5$, we have
$$
\pm T_{\mathrm{C}}(\pm p) = (-1)^{\frac{p^{2}-1}{24}}+\frac{1}{2p} \sum_{j=0}^{p-1}\sum_{n=0}^{p-1} (-1)^{n}\zeta_{p}^{\left(n+1-\frac{p^{2}-1}{24}\right)j}(1-\zeta_{p}^{j})\cdots(1-\zeta_{p}^{nj})
$$
where $p\equiv \pm 1\Mod{6}$. 
\end{corollary}
\begin{proof}
From $T_{p}^{\infty} f_{\mathrm{C}}(0) = \pm T_{\mathrm{C}}(\pm p) f_{\mathrm{C}}(0) = \pm 2T_{\mathrm{C}}(\pm p )$ with $f_{\mathrm{C}}(0)=2$, we have
\begin{align*}
\pm T_{\mathrm{C}}(\pm p) &= \frac{1}{2} \left(2(-1)^{\frac{p^{2}-1}{24}} + \frac{1}{p}\sum_{j=0}^{p-1}\zeta_{24}^{-pj} f_{\mathrm{C}}\left(\frac{j}{p}\right)\right) \\
&=(-1)^{\frac{p^{2}-1}{24}} + \frac{1}{2p} \sum_{j=0}^{p-1} \zeta_{24}^{-pj} \zeta_{24p}^{j}\left(1+\sum_{n=0}^{p-1}(-1)^{n}\zeta_{p}^{(n+1)j}(1-\zeta_{p}^{j})\cdots(1-\zeta_{p}^{nj})\right) \\
&= (-1)^{\frac{p^{2}-1}{24}}+\frac{1}{2p}\sum_{j=0}^{p-1} \zeta_{p}^{-\frac{p^{2}-1}{24}j} + \frac{1}{2p}\sum_{j=0}^{p-1}\sum_{n=0}^{p-1}(-1)^{n}\zeta_{p}^{(n+1-\frac{p^{2}-1}{24})j}(1-\zeta_{p}^{j})\cdots(1-\zeta_{p}^{nj})\\
&= (-1)^{\frac{p^{2}-1}{24}}+ \frac{1}{2p}\sum_{j=0}^{p-1}\sum_{n=0}^{p-1}(-1)^{n}\zeta_{p}^{(n+1-\frac{p^{2}-1}{24})j}(1-\zeta_{p}^{j})\cdots(1-\zeta_{p}^{nj}).
\end{align*}
Here we use $\sum_{j=0}^{p-1}\zeta_{p}^{-\frac{p^{2}-1}{24}j} =0$, which follows from $p\nmid \left(\frac{p^{2}-1}{24}\right)$. 
\end{proof}
For example, for $p=73, 97$, we have $T_{\mathrm{C}}(73)=2, T_{\mathrm{C}}(97)=-2$ and  we have
\begin{align*}
2&=1+\frac{1}{146}\sum_{j=0}^{72}\sum_{n=0}^{72} (-1)^{n}\zeta_{73}^{(n-221)j} (1-\zeta_{73}^{j})\cdots (1-\zeta_{73}^{nj}) \\
-2&=1+\frac{1}{194}\sum_{j=0}^{96}\sum_{n=0}^{96} (-1)^{n}\zeta_{97}^{(n-391)j} (1-\zeta_{97}^{j})\cdots (1-\zeta_{97}^{nj})
\end{align*}

We can apply the same argument for the Li-Ngo-Rhoades' example. If we consider the value $T_{p}^{\infty} f_{\mathrm{L}}(0)$, we obtain the following formula of $T_{\mathrm{L}}(p)$. 
\begin{corollary} For any prime $p\geq 3$, we have
$$
\pm T_{\mathrm{L}}(\pm p) = 1 + \frac{1}{p}\sum_{j=0}^{p-1} \zeta_{p}^{-\frac{p^{2}-1}{8}j} \sum_{n=0}^{(p-1)/2} \frac{(1-\zeta_{p}^{j})(1-\zeta_{p}^{3j})\cdots (1-\zeta_{p}^{(2n-1)j})(-\zeta_{p}^{j})^{n}}{(1+\zeta_{p}^{2j})(1+\zeta_{p}^{4j})\cdots (1+\zeta_{p}^{2nj})}
$$
where $p\equiv \pm 1\Mod{4}$.  
\end{corollary}
For example, when $p=7,31$, we have $T_{\mathrm{L}}(-7) = -2,  T_{\mathrm{L}}(-31) = 2$ and 
\begin{align*}
2&= 1+ \frac{1}{7}\sum_{j=0}^{6} \zeta_{7}^{-6j} \sum_{n=0}^{3} \frac{(1-\zeta_{7}^{j})(1-\zeta_{7}^{3j})\cdots (1-\zeta_{7}^{(2n-1)j})(-\zeta_{7}^{j})^{n}}{(1+\zeta_{7}^{2j})(1+\zeta_{7}^{4j})\cdots (1+\zeta_{7}^{2nj})} \\
-2&= 1 + \frac{1}{31}\sum_{j=0}^{30} \zeta_{31}^{-120j} \sum_{n=0}^{15} \frac{(1-\zeta_{31}^{j})(1-\zeta_{31}^{3j})\cdots (1-\zeta_{31}^{(2n-1)j})(-\zeta_{31}^{j})^{n}}{(1+\zeta_{31}^{2j})(1+\zeta_{31}^{4j})\cdots (1+\zeta_{31}^{2nj})} 
\end{align*}

%%%%%%%%%%%%%%

%\begin{acknowledgements}
%\end{acknowledgements}



\end{document}